\documentclass[reqno, 11pt]{amsart}

\usepackage[letterpaper,hmargin=1in,vmargin=1.2in]{geometry}
\usepackage{amsmath, amssymb, amsthm, verbatim, url}
\usepackage{graphicx}
\usepackage{amsmath,amssymb,amsthm,mathrsfs,graphicx,url}
\usepackage[usenames,dvipsnames]{color}
\usepackage[colorlinks=true,linkcolor=Black,citecolor=Black,urlcolor=Black]{hyperref}

\DeclareMathOperator \re {Re}
\DeclareMathOperator \im {Im}

\newtheorem*{lem}{Lemma}

\newtheorem{thm}{Theorem}

\theoremstyle{definition}

\numberwithin{prop}{section}

\begin{document}

\title{Semiclassical estimates for scattering on the real  line}

\author{Kiril Datchev}
\address{Department of Mathematics, Purdue University,  West Lafayette, IN, USA}
\email{kdatchev@purdue.edu}
\author{Jacob Shapiro}
\address{Mathematical Sciences Institute, Australian National University, Acton, ACT, Australia}
\email{jacob.shapiro@anu.edu.au}

\thanks{The authors are grateful to Maciej Zworski for his suggestions, comments, and encouragement. KD was partially supported by NSF Grant DMS-1708511, and JS was partially supported by an AMS-Simons Travel Grant and by ARC grant DP180100589.}

\begin{abstract}
 We prove explicit semiclassical resolvent estimates for an integrable potential on the real line. The proof is a comparatively easy case of the spherical energies method, which has been used to prove similar theorems in higher dimensions and in more complicated geometric situations. The novelty in our results lies in the weakness of the assumptions on the potential.
\end{abstract}

\maketitle

\section{Introduction}

In this paper we prove  explicit resolvent estimates for the Schr\"odinger operator
\[
 P = -h^2 \partial_x^2 + V,
\]
where $h>0$ is a semiclassical parameter, and $V\colon \mathbb R \to \mathbb R$ is integrable.

\begin{thm}\label{t:adelta}
Let $V \in L^1(\mathbb R)$. Then 
\begin{equation}\label{e:expm}
 \|m^{1/2} (P  - E-i\varepsilon)^{-1} m^{1/2}\|_{L^2(\mathbb R) \to L^2(\mathbb R)} \le \exp\left(\frac{2 E^{-1/2} \int m }{ h }\right),
\end{equation}
for any $m \in L^1(\mathbb R)$ such that
$|V|\le m$, and for any $E \ge 2 \varepsilon>0$ and $h>0$.
\end{thm}
For example, if $|V| \le A (1+|x|)^{-1-\delta}$, then we get a bound between the usual weighted spaces of the limiting absorption principle $(1 + |x|)^{- \frac {1 + \delta}2}L^2(\mathbb R) \to (1 + |x|)^{\frac {1 + \delta}2} L^2(\mathbb R)$; see \cite[\S XIII.8]{rs}. We have not attempted to optimize the numerical constants, and the same proof gives   similar bounds in other sectors $E \ge c \varepsilon >0$. The estimate is invariant under two rescalings of the operator, one is $P = h^2(-\partial_x^2 + h^{-2}V)$ and the other is $x = hy$.

Estimates like \eqref{e:expm} are important for their applications to smoothing, wave decay, and resonance free regions; see \cite[Chapter 6]{dz} for an introduction and \cite[\S3.2]{zs} for a survey of recent results. When $V$ is compactly supported, we also have an improvement away from the support of $V$ in Theorem \ref{one d resolv est} below.

The main novelty lies in the sharp dependence on $\varepsilon$ and $h$ (see \cite{ddz, dj} for corresponding lower bounds), under   weak regularity and decay assumptions on $V$. The decay assumption is essentially optimal; examples due to Wigner and Von Neumann show that $P$ may have a positive eigenvalue if $V$ decays like $|x|^{-1}$ \cite[\S XIII.13]{rs}, and in such a situation $m^{1/2} (P  - E-i\varepsilon)^{-1} m^{1/2}$ is unbounded as $\varepsilon \to 0$. Bounds for more slowly decaying $V$ hold under assumptions on $V'$ \cite{da}.

Our proof is a comparatively easy case of the method of estimating spherical energies, which has been used to prove versions of \eqref{e:expm} in more complicated geometric situations, but with stronger regularity and decay assumptions. For semiclassical resolvent estimates, this method goes back to the work of Cardoso and Vodev \cite{cv}, and more generally in scattering theory it goes back  to the work of Kato \cite{k1}.
In our setting, among other simplifications, in place of a spherical energy we have the pointwise energy
 \begin{equation}\label{e:fdef}
  F(x) = |hu'(x)|^2 + E |u(x)|^2.
 \end{equation}
A similar pointwise energy is used implicitly in \cite[\S2]{cd} to prove uniform resolvent estimates for repulsive potentials on the half-line.

When $V$ is smooth, the exponential bound \eqref{e:expm} (with different constants) was first proved by Burq \cite{burq0, burq}, who also considered higher dimensional problems and other generalizations. Further proofs and generalizations can be found in \cite{cv,rt,v,da,ddh,sh, gannot}. If $V \in L^\infty(\mathbb R^n)$ with $n>1$, then only weaker versions of \eqref{e:expm} are known \cite{sh2, kv, v2, v3, v4}, with $e^{M/h}$ replaced by $e^{M/h^\rho}$ with $\rho >1$, even if $V$ has compact support.

If $V$ is compactly supported (or, more generally, holomorphic near infinity), then the bound \eqref{e:expm} implies the existence of an exponentially small resonance free region near the real axis, thanks to an identity of Vodev \cite[(5.4)]{v}. For $V \in L^\infty(\mathbb R)$ compactly supported, such regions have been known to exist for some time: see \cite{h}, and also \cite[\S 2.8]{dz} for another proof as well as detailed examples and more references. But the reverse problem of deducing a resolvent estimate like \eqref{e:expm} from the existence of a resonance free region seems to be more difficult.

To the authors' knowledge, \eqref{e:expm} is the first semiclassical resolvent estimate for $V \not\in L^\infty(\mathbb R)$, but there has been much work on related problems. Zworski \cite{z} and Hitrik \cite{hit} analyzed the distribution of resonances for $V \in L^1(\mathbb R)$ either compactly supported or exponentially decaying, using  Melin's representation of the scattering matrix \cite{m}. More recently Korotyaev (see \cite{kor} and references therein) has proved many further results in this topic and other related ones, including trace formulas and inverse results. Note however that the methods in those papers require at least $(1+|x|)V \in L^1(\mathbb R)$, due to the finer aspects of scattering theory being analyzed, whereas in the present paper we require only $V \in L^1(\mathbb R)$. The condition that $V \in L^1(\mathbb R)$ is called the short range condition, and it allows one to extend the integral kernel of the resolvent up to the continuous spectrum, while as mentioned above if $V \not\in L^1(\mathbb R)$ then the continuous spectrum may contain embedded eigenvalues: see \cite[\S 5.1]{y} and \cite{y18} and references therein for more on this and for other results concerning short and long range potential scattering in one dimension.

When $V$ is compactly supported, we have an improvement away from the support of $V$.

\begin{thm} \label{one d resolv est}
Let $V \in L^1(\mathbb R)$ be  supported in $[-R,R]$.
Then 
\begin{equation}\label{e:ext}
 \|{\bf 1}_{|x|>R}(1 + |x|)^{- \frac {1 + \delta}2}(P - E-i\varepsilon)^{-1}(1+|x|)^{- \frac {1 + \delta}2}{\bf 1}_{|x|>R}\|_{L^2(\mathbb R) \to L^2(\mathbb R)} \le \frac{8(1+R)^{-\delta}\delta^{-1} E^{- 1/2}}h,
\end{equation}
for any $E \ge 2 \varepsilon>0$ and $\delta, \ h >0$, where ${\bf 1}_{|x|>R}$ is the characteristic function of $\{x \colon |x|>R\}$.
\end{thm}

When $V$ is smooth, the improvement \eqref{e:ext}  away from the support of $V$ was first proved by Cardoso and Vodev \cite{cv}, refining earlier work of Burq \cite{burq}, and again analogous results  hold for much more general operators \cite{cv,rt,v,da,ddh,sh}. But if $n>1$, then the cutoff ${\bf 1}_{|x|>R}$ may need to be replaced by ${\bf 1}_{|x|>R'}$ with $R' \gg R$, even when $V \in C_c^\infty(\mathbb R^n)$; see \cite{dj} for corresponding lower bounds, and also for an application of an exterior estimate like \eqref{e:ext} to integrated wave decay.

The rest of the paper is organized as follows. In \S\ref{s:proofs}, we prove a stronger weighted resolvent estimate which implies both \eqref{e:expm} and \eqref{e:ext}. In \S\ref{s:wronskian}, we  prove \eqref{e:ext} in the case $\varepsilon = 0$ by estimating  the integral kernel of the resolvent using Wronskian identities; it would be interesting to know if such an approach could be applied to \eqref{e:expm}. Throughout, $L^p$ means $L^p(\mathbb R)$.

\section{Weighted resolvent estimates}\label{s:proofs}

We will deduce Theorems \ref{t:adelta} and 2 from the following stronger result.

\begin{thm}\label{t:m}
 Let $V \in L^1$, let $h>0$, and  $E\ge 2 \varepsilon > 0$. Fix $w\colon \mathbb R \to \mathbb [-1,1]$ such that $w' \in L^1$ and
\begin{equation}\label{e:w2}
\frac k h |Vw|\le w',
\end{equation}
where $k = 4/E^{1/2}$. Then 
\begin{equation}\label{e:resestw}
 \|(w')^{1/2} (P - E-i\varepsilon)^{-1}(w')^{1/2}\|_{L^2 \to L^2} \le \frac {8E^{-1/2}}h.
\end{equation} 
\end{thm}
Note that  $w$ may depend on $h$, $V$, and $E$. As a simpler first case, the reader can consider $V \equiv 0$ and $w(x) = 1 - (1 + x)^{-\delta}$, $x > 0$, $\delta > 0$, and $w$ odd. A variant of this is used in the proof of Theorem 2 below.

To prove Theorem \ref{t:m}, we will need the following  essentially well-known lemma. 

\begin{lem}
Let $\mathcal D$ be the set of all $u \in L^2 \cap L^\infty$ such that $u' \in  L^2 \cap L^\infty$ and $Pu \in L^2$.  Then $P$ is self-adjoint on $L^2$ with domain $\mathcal D$.
 In particular, $P-z$ is bijective from $\mathcal D$ to $L^2$ for all $z \in \mathbb C \setminus \mathbb R$.
\end{lem}

\begin{proof}[Proof of Lemma]

Let $\mathcal D_\textrm{max}$ be the set of all $u \in L^2$ such that $u' \in L^1_\textrm{loc}$ and $Pu \in L^2$. We  begin by proving that  $\mathcal D_\textrm{max} = \mathcal D$. Indeed, for any $a>0$ and $u \in \mathcal D_\textrm{max}$, by integration by parts and Cauchy--Schwarz we have
\begin{equation*}
\begin{gathered}
\int_{-a}^a|u'|^2 = u'\bar u|_{-a}^a - \int_{-a}^a u''\bar u \le 2\sup_{[-a,a]}|u'|\sup_{[-a,a]} |u| + h^{-2}\|V\|_{L^1} \sup_{[-a,a]}|u|^2 + h^{-2}\|Pu\|_{L^2}\|u\|_{L^2},\\
\sup_{[-a,a]}|u|^2 =\sup_{x \in [-a,a]}  \left(|u(0)|^2 + 2 \re \int_0^x u'\bar u\right)  \le |u(0)|^2 + 2\left(\int_{-a}^a|u'|^2\right)^{1/2}\|u\|_{L^2},\\
\sup_{[-a,a]} |u'|^2 \le |u'(0)|^2 + 2 h^{-2}\|V\|_{L^1}\sup_{[-a,a]}|u|  \sup_{[-a,a]}|u'| +  2h^{-2}\|Pu\|_{L^2}\left(\int_{-a}^a|u'|^2\right)^{1/2}. 
\end{gathered}
\end{equation*}
This is a system of inequalities of the form $x^2 \le 2yz+Ay^2 + B$, $y^2 \le C + Dx$, $z^2 \le E + Fyz + Gx$. After using the second to eliminate $y$, we obtain a system in $x$ and $z$ with quadratic left hand sides and subquadratic right hand sides. Hence $x$, $y$, and $z$ are each bounded in terms of $A, B, \dots, G$. Letting $a \to \infty$, we conclude that $u' \in L^2$, $u \in L^\infty$, and $u' \in L^\infty$. Hence  $\mathcal D_\textrm{max} = \mathcal D$.

Equip $P$  with the domain $\mathcal D_\textrm{max} = \mathcal D \subset L^2$. By integration by parts, $P \subset P^*$. But, by Sturm--Liouville theory,  $P^* \subset P$: see  \cite[\S 3.B]{w}, \cite[\S17.4]{n}, or \cite[Lemma 10.3.1]{zet}. Hence $P=P^*$.
\end{proof}

\begin{proof}[Proof of Theorem \ref{t:m}]
Since $L^2\cap L^\infty$ is dense in $L^2$, it is enough to prove
\begin{equation}\label{e:apr}
   \frac {E}{7} \int w'|u|^2 +  \frac 1 6 \int w' |hu'|^2 \le \frac {15} {2h^2} \int |v|^2, \quad  \textrm{ for all }v \in L^2 \cap L^\infty,
\end{equation}
where $u = (P-E-i\varepsilon)^{-1}(w')^{1/2}v$. 

To prove \eqref{e:apr}, define $F$ by \eqref{e:fdef}. By the Lemma, $F' \in L^1$ and is given by 
\[
 F' = - 2 (w')^{1/2}\re v \bar u'  + 2 V \re u \bar u'- 2 \re i \varepsilon u \bar u ',
\]
and we have
\[
 (wF)' = - 2 w(w')^{1/2}\re v \bar u'   + 2 wV \re u \bar u'- 2 w \re i \varepsilon u \bar u ' + w'|hu'|^2 + E w'|u|^2.
\]
Using $|w| \le 1$ and \eqref{e:w2} gives
\begin{equation}\label{e:wf'}
 (wF)' \ge - 2(w')^{1/2}|vu'| - \frac {2h}k w'|uu'| - 2 \varepsilon|wuu'| + w'|hu'|^2 + E w'|u|^2.
\end{equation}
Observe now that, by the Lemma,  $wF$ and $(wF)'$ are in $L^1$, so that
$\int (wF)' = 0$, and hence
\begin{equation}\label{e:epsright}
E \int w'|u|^2  +  \int w' |hu'|^2  \le 2\int (w')^{1/2}|vu'| + \frac {2h}k \int  w'|uu'|+ 2  \varepsilon \int |wuu'|.
\end{equation}
The first two terms on the right contain $w'$, and that will make them easy to handle. The last term requires more work and we begin by showing how to estimate it using integrals containing $w'$. By Cauchy--Schwarz and integration by parts we have
\begin{equation}\label{e:csep}
 2\varepsilon \int |wuu'|  \le 2\left(  \int \varepsilon |u|^2 \int \varepsilon w^2|u'|^2\right)^{1/2}, \quad 
 \int \varepsilon |u|^2 = - \im \int (w')^{1/2}v \bar u \le \int(w')^{1/2}|vu|.
\end{equation}
Similarly,
\[
\begin{split}
  \int \varepsilon w^2|u'|^2 &\le \int \varepsilon |w||u'|^2 = -  \re \int \varepsilon |w|'u'\bar u - \re\int \varepsilon |w| u'' \bar u \\
  &=  - \re \int \varepsilon |w|'u'\bar u + \re\int \frac \varepsilon {h^2} |w| (w')^{1/2} v \bar u + \int \frac{\varepsilon E}{h^2} |w||u|^2 - \int \frac{\varepsilon}{h^2}|w| V|u|^2\\
  &\le \varepsilon \int w'|uu'| + \frac \varepsilon {h^2}\int(w')^{1/2}|vu| + \frac E {h^2} \int (w')^{1/2}|vu| + \frac \varepsilon {hk}\int w'|u|^2,
\end{split}
\]
where we used $|w|u'\bar u \in L^1$, $(|w|u'\bar u)' \in L^1$, $|w|\le1$, the second of \eqref{e:csep}, and \eqref{e:w2}. 
Substituting into the first of \eqref{e:csep} gives
\[\begin{split}
  2\varepsilon \int |wuu'| &\le 2\left(\int (w')^{1/2}|vu|\left( \varepsilon\int w'|uu'| + \frac {E+ \varepsilon}{h^2} \int (w')^{1/2}|vu| + \frac \varepsilon {hk}\int w'|u|^2) \right) \right)^{1/2}\\
  &\le \frac {\varepsilon h} {(E+ \varepsilon)^{1/2}} \int w'|uu'| + \frac {2(E+ \varepsilon)^{1/2}}{h} \int (w')^{1/2}|vu| + \frac{\varepsilon }{k(E+ \varepsilon)^{1/2}} \int w'|u|^2,
\end{split}\]
where we used  $2\sqrt{ab} \le a + b$ with $a=\int(w')^{1/2}|vu|(E+\varepsilon)^{1/2} /h$. Using the bounds $\varepsilon \le E/2$ and  $\varepsilon/(E+\varepsilon)^{1/2} \le (E/6)^{1/2}$, and
substituting into \eqref{e:epsright} gives
\[\begin{split}
\left( E  - \frac{E^{1/2}}{6^{1/2}k}\right)&\int w'|u|^2  +  \int w' |hu'|^2  \le \\ &2\int (w')^{1/2}|vu'| + \left(\frac {2h}k + \frac{E^{1/2}h}{6^{1/2}}\right)\int  w'|uu'|+  \frac {(6E)^{1/2}}{h} \int (w')^{1/2}|vu|.
\end{split}\]
Now we estimate the three terms on the right using Cauchy--Schwarz and $2\sqrt{ab} \le a + b$, balancing the constants so that all terms with $u$ can be absorbed back into the left side.
\[\begin{split}
  2\int (w')^{1/2}|vu'| &\le \frac 3 {h^2} \int  |v|^2 + \frac 13 \int w'|hu'|^2,\\
 \left(\frac {2h}k + \frac{E^{1/2}h}{6^{1/2}}\right) \int  w'|uu'| &\le  \frac 12 \left(\frac {2}k + \frac{E^{1/2}}{6^{1/2}}\right) ^2\int w'|u|^2  + \frac 12 \int w'|hu'|^2,\\
\frac {(6E)^{1/2}}{h}\int (w')^{1/2}|vu| &\le \frac {9} {2h^2} \int|v|^2  + \frac E 3 \int w'|u|^2,\\
\end{split}\]
giving
\[
 \left( \frac {2E}3  - \frac {E^{1/2}} {6^{1/2}k} - \frac 12 \left(\frac {2}k + \frac{E^{1/2}}{6^{1/2}}\right) ^2\right)\int w'|u|^2  +  \frac 16\int w' |hu'|^2  \le \frac {15} {2h^2} \int |v|^2,
\]
which implies \eqref{e:apr}. 
\end{proof}

\begin{proof} [Proof of Theorem \ref{t:adelta}]
Let
\[
 w(x) = 2\exp\left(-\frac k {2h} \int m\right)\sinh\left(\frac k h \int_a^x m\right),
\]
with $a \in \mathbb R$ chosen such that $\int_a^\infty m = \int_{-\infty}^a m$. Then substituting
\[
 w'(x) \ge \frac {2k}h \exp\left(-\frac k {2h} \int m\right)m(x)
\]
into \eqref{e:resestw} gives \eqref{e:expm}.
\end{proof}

\begin{proof}[Proof of Theorem \ref{one d resolv est}]
We apply \eqref{e:resestw} with $w$ an odd function vanishing on $[-R,R]$ and obeying
\[
w(x) = 1 - \frac{(1+R)^\delta}{(1+x)^{\delta}}, \quad \textrm{ when } x > R,
\]
and use $w'(x) = \delta(1+R)^{\delta}(1+|x|)^{-1-\delta} \textbf{1}_{|x|>R}$.
\end{proof}

\section{Integral kernel estimates}\label{s:wronskian}
We conclude by proving  \eqref{e:ext} for $\varepsilon = 0$ in another way, by estimating the integral kernel of $(P-E-i0)^{-1}$. By direct calculation (by the method of variation of parameters, or see \cite[(1.28)]{tz} or \cite[Proposition 5.1.4]{y}) this integral kernel is given by
\begin{equation}\label{e:kdef}
 K(x,y) = -u_-(x)u_+(y)/h^2W(u_-,u_+), \qquad \textrm{for } x \le y,
\end{equation}
and it obeys $K(x,y) = K(y,x)$, where $(P-E)u_\pm=0$,
\[
 u_+(x) = 
 \begin{cases}e^{ix\lambda} &\textrm{ when } x >R,\\
  A e^{ix \lambda} + B e^{-ix \lambda}, &\textrm{ when } x <-R,
 \end{cases}
 \quad  
 u_-(x) = 
  \begin{cases} C e^{ix \lambda} + D e^{-ix \lambda}&\textrm{ when } x >R,\\
  e^{-ix \lambda}, &\textrm{ when } x <-R,
 \end{cases}
\]
where $\lambda = E^{1/2}/h$, and  $W(u_-,u_+) = u_-u_+' - u_-'u_+$ is the Wronskian. The solutions $u_\pm$ are (multiples of) distorted plane waves or Jost solutions, and they are also used to define the scattering matrix. See \cite[\S2.4]{dz} for an introduction, and \cite[\S5.1.1]{y} for more on $u_\pm$ and $K$ when $V$ is not necessarily compactly supported.

Since the Wronskian is independent of $x$, we compute it when $\pm x>R$ and equate to obtain
\[
 W(u_-,u_+) = 2A i\lambda = 2D i\lambda \qquad \Longrightarrow \qquad A=D.
\]
Repeating the above with $ W(\overline{u_+},u_+)$,  $W(\overline{u_-},u_-)$, and $W(\overline{u_-} ,u_+)$, 
gives
\[
  1 = |A|^2 - |B|^2, \qquad -1 = |C|^2-|D|^2, \qquad -B = \overline{C},
\]
(which is equivalent to unitarity of the scattering matrix).
We conclude that when $|x|>R$ and $|y|>R$ we have
\begin{equation}
 |K(x,y)| \le \frac{|A|+|B|}{2|A|hE^{1/2}} \le \frac 1 {hE^{1/2}}.
\end{equation}
Combining this with $
\int_R^\infty(1+y)^{-\frac{1+\delta}2}f(y)dx \le  \|f\|_{L^2(\mathbb R)}\delta^{-1/2}(1+R)^{-\delta/2}
$
gives 
\[
  \|{\bf 1}_{|x|>R}(1 + |x|)^{- \frac {1 + \delta}2}(P - E-i0)^{-1}(1+|x|)^{- \frac {1 + \delta}2}{\bf 1}_{|x|>R}\|_{L^2 \to L^2} \le \frac {2} {h \delta (1+R)^\delta E^{1/2}}.
\]

\end{document}